\documentclass[a4paper,10pt]{amsart}
\usepackage{amsthm}
\usepackage{amssymb}
\usepackage{amsfonts}
\usepackage{mathrsfs}
\usepackage{mathtools}

\newtheorem{theorem}{Theorem}[section]
\newtheorem{proposition}[theorem]{Proposition}
\newtheorem{lemma}[theorem]{Lemma}

\theoremstyle{definition}
\newtheorem{definition}[theorem]{Definition}

\newtheorem{remark}[theorem]{Remark}

\newcommand{\Mg}[1]{\mathcal{M}_{g,#1}}
\newcommand{\Mgn}{\mathcal{M}_{g,n}}
\newcommand{\Mgnbar}{\overline{\mathcal{M}}_{g,n}}
\newcommand{\Mgbar}[1]{\overline{\mathcal{M}}_{g,#1}}
\newcommand{\Mbar}[1]{\overline{\mathcal{M}}_{#1}}
\renewcommand{\d}{\underline{d}}
\newcommand{\C}{\mathcal{C}}
\newcommand{\D}{\overline{D}}
\newcommand{\Dd}{\D_{\d}}
\newcommand{\Dscr}{\mathscr{D}}
\newcommand{\E}{\mathbb{E}}
\newcommand{\Kscr}{\mathscr{K}}
\renewcommand{\L}{\mathscr{L}}
\renewcommand{\O}{\mathcal{O}}
\renewcommand{\P}{\mathbb{P}}
\newcommand{\Q}{\mathbb{Q}}
\newcommand{\U}{\mathcal{U}}
\newcommand{\Wg}{\overline{\mathcal{W}}_g}
\renewcommand{\phi}{\varphi}

\newcommand{\abs}[1]{\left| #1 \right|}
\newcommand{\linsys}[1]{\big| #1 \big|}

\DeclareMathOperator{\ch}{ch}
\DeclareMathOperator{\td}{td}
\DeclareMathOperator{\ord}{ord}
\DeclareMathOperator{\Pic}{Pic}

\title{The pullback of a Theta divisor to $\Mgnbar$}
\author{Fabian M\"uller}
\address{Humboldt-Universit\"at zu Berlin, Institut f\"ur Mathematik, 10099
Berlin}
\email{muellerf@math.hu-berlin.de}

\begin{document}

\begin{abstract}
We compute the class of a divisor on $\Mgnbar$ given as the closure of the
locus of smooth pointed curves $\left[ C;\, x_1,\, \dots,\, x_n \right]$ for
which $\sum d_j x_j$ has an effective representative, where $d_j$ are
integers summing up to $g-1$, not all positive. The techniques used are a
vector bundle computation, a pushdown argument reducing the number of marked
points, and the method of test curves.
\end{abstract}

\maketitle

\section{Introduction}
\label{sec:introduction}
It has long been known classically that if $C$ is a smooth curve of genus
$g \geq 2$ and $C_{g-1}$ denotes its $(g-1)$-fold symmetric product, the
Abelian sum map $C_{g-1} \to \Pic^{g-1}(C)$, which to $g-1$ unordered points
$x_1,\, \dots,\, x_{g-1}$ associates the line bundle $\O_C(x_1 + \dots +
x_{g-1})$, has as image a divisor, which becomes a theta divisor under an
identification of $\Pic^{g-1}(C)$ with the Jacobian of $C$. This result can be
globalized to a map $\C_{g,g-1} \to \Pic_g^{g-1}$, where
\begin{equation*}
\C_{g,g-1} = \left( \Mg{1} \times_{\mathcal{M}_g} \dots \times_{\mathcal{M}_g}
\Mg{1} \right) \big/ S_{g-1}
\end{equation*}
is the $(g-1)$-fold symmetric product of the universal curve, and
$\Pic_g^{g-1}$ is the universal Picard variety of degree $g-1$. The image of
this map is again a divisor, which we denote by $\Theta_g$. Given an integer
vector $\d = (d_1,\, \dots,\, d_n) \in \mathbb{Z}^n$ satisfying $\sum_{j=1}^n
d_j = g - 1$, we can define a map $\phi_{\d} \negthinspace:\, \Mgn \to
\Pic_g^{g-1}$ by associating to a pointed curve $\left[ C;\, x_1,\, \dots,\,
x_n \right]$ the line bundle $\O_C(d_1 x_1 + \dots + d_n x_n)$ on $C$. If at
least one of the $d_j$ is negative the image of $\phi_{\d}$ is not contained
in $\Theta_g$, and we can ask what is the class of the pullback $D_{\d} :=
\phi_{\d}^* \Theta_g$ and its closure on $\Mgnbar$. Unraveling the concepts
involved, we arrive at the following equivalent definition:

\begin{definition}
Let $\d = (d_1,\, \dots,\, d_n)$ be an $n$-tuple of integers satisfying
$\sum_{j=1}^n d_j = g - 1$, with at least one $d_j$ negative. Denote by
\begin{equation*}
D_{\d} := \left\{ \left[ C;\, x_1,\, \dots,\, x_n \right] \in \Mgn \,\Big|\,
h^0\left(C,\, d_1 x_1 + \dots + d_n x_n \right) \geq 1 \right\},
\end{equation*}
which is a divisor on $\Mgn$, and let $\Dd$ be its closure in $\Mgnbar$.
\end{definition}

Note that since the $x_j$ are distinct, the condition $h^0\left(C,\, d_1 x_1 +
\dots + d_n x_n \right) \geq 1$ is equivalent to postulating that there is a
pencil of degree $d_{S_+} := \sum_{j:d_j>0} d_j$ on $C$ that contains the
divisor $\sum_{j: d_j > 0} d_j x_j$ and has a section that vanishes to order
$-d_j$ at $x_j$ for all $j \in S_- := \big\{ j \,\big|\, d_j < 0 \big\}$. As
it ties in nicely with the limit linear series characterization on reducible
curves, we will always use this reformulation from now on.

The main result of this paper, which is proven in Theorem \ref{thm:class_Dd},
is the computation of the class of this divisor in $\Pic(\Mgnbar)$. It is
given by
\begin{equation}
\label{eq:class_Dd}
\begin{split}
\left[ \Dd \right] =& -\lambda + \sum_{j=1}^n \binom{d_j + 1}{2} \psi_j -
0 \cdot \delta_0 \\
& - \sum_{\substack{i,\, S\\ S \subseteq S_+}}
\binom{\abs{d_S - i} + 1}{2} \delta_{i:S} - \sum_{\substack{i,\, S\\ S
\not\subseteq S_+}} \binom{d_S - i + 1}{2} \delta_{i:S},
\end{split}
\end{equation}
where $S_+ := \big\{ j \,\big|\, d_j > 0 \big\}$ and $d_S := \sum_{j \in S}
d_j$. Thus the next to last summand corresponds to boundary classes that
parameterize reducible curves where the points indexed by $S_-$ lie on a
single component, while the last one corresponds to classes parameterizing
curves which have points from $S_-$ on both components.

In the special case $\d = (d_1,\, \dots,\, d_{n-1};\, -1)$ with $d_1,\,
\dots,\, d_{n-1} > 0$, the divisor $\Dd$ is just the pullback to $\Mgnbar$ of
the divisor of pointed curves $\left[ C;\, x_1,\, \dots,\, x_{n-1} \right] \in
\Mgbar{n-1}$ having a $g^1_g$ containing $d_1 x_1 + \dots + d_{n-1} x_{n-1}$,
which was considered by A.~Logan \cite{bib:logan}. For $n = 2$, it is the
pullback of the Weierstra\ss{} divisor on $\Mgbar{1}$, whose class has been
computed by F.~Cukierman \cite{bib:cukierman} to be
\begin{equation}
\label{eq:weierstrass_divisor_class}
\left[ \Wg \right] = -\lambda + \binom{g + 1}{2} \psi_1 - \sum_{i=1}^{g-1}
\binom{g - i + 1}{2} \delta_{i:1}.
\end{equation}
For more details on this, see Remarks \ref{rmk:weierstrass_divisor} and
\ref{rmk:logan_divisor}.

A divisor similar to $\Dd$ was studied by R.~Hain \cite{bib:hain}: On an open
subset $U$ of $\Mgnbar$ (or a covering of such) where there is a globally
defined theta characteristic $\alpha$, one can define a morphism $\phi_{\d}'
\negthinspace:\, U \to \Pic_g^0$ mapping a pointed curve $\left[ C;\, x_1,\,
\dots,\, x_n \right]$ to the line bundle $\O_C(d_1 x_1 + \dots + d_n x_n -
\alpha) \in \Pic^0(C)$. The class of the closure in $\Mgnbar$ of the pullback
$D_{\d}' := \big(\phi_{\d}' \big)^* \Theta_\alpha$ is computed in
\cite[Theorem 11.7]{bib:hain}; expressed in our notation it is
\begin{equation*}
\left[ \Dd' \right] = -\lambda + \sum_{j=1}^n \binom{d_j + 1}{2} \psi_j +
\delta_0/8 - \sum_{i,S} \binom{d_S - i + 1}{2} \delta_{i:S} \in \Pic(\Mgnbar)
\otimes \Q.
\end{equation*}
Both this result and our Theorem \ref{thm:class_Dd} are reproven in a recent
preprint by S.~Grushevsky and D.~Zakharov \cite[Theorem
6]{bib:grushevsky-zakharov}, where it is also shown that the divisor
considered by Hain is reducible and decomposes as $\Dd$ together with some
boundary components, with multiplicities according to the generic vanishing
order of the theta function.

This paper is organized as follows: In Section \ref{sec:preliminaries} we will
collect some results on pullbacks and pushforwards of divisors on $\Mgnbar$
that we will need during the course of the paper. In Section
\ref{sec:main_coefficients} the coefficients of the $\lambda$ and $\psi_j$
classes in the expression for $\left[ \Dd \right]$ are computed by a vector
bundle technique. The rest of the coefficients are computed via test curves.
The actual test curve computations are done in Section \ref{sec:test_curves},
and the results are applied in Section \ref{sec:boundary} together with a
pushdown technique to finish the proof of the main result.

\subsection*{Notation}
By a \emph{nodal curve}, we shall mean a reduced connected 1-dimensional
scheme of finite type over a field $k$ whose only singularities are ordinary
nodes. A nodal curve is said to be of \emph{compact type} if its dual graph
is a tree, or equivalently if its Jacobian is compact.

We use the shorthand $[n] := \{ 1,\, \dots,\, n \}$. If $a$ is any
expression, we write $(a)_+ := \max(a,\, 0)$. Occasionally we will write down
a binomial coefficient $\binom{a}{2}$ with $a < 0$, by which we just mean $a
(a - 1) / 2$.

If $\d = (d_1,\, \dots,\, d_n)$ is an $n$-tuple of integers, we write $S_+$
(resp. $S_-$) for the set of indices $j \in [n]$ with $d_j > 0$ (resp. $d_j <
0$). Moreover, if $S \subseteq [n]$ is an arbitrary set of indices, we write
$d_S := \sum_{j \in S} d_j$. When convenient, we will assume that the
positive $d_j$ come first and in the notation $\Dd$ separate them with a
semicolon from the negative ones.

When summing over boundary classes $\delta_{i:S}$ in $\Pic(\Mgnbar)$, the
summation range $\sum_{i,S}$ (and obvious analogues) will be implicitly taken
to involve only admissible combinations (e.~g. $\abs{S} \geq 2$ for $i = 0$)
and to contain every divisor only once (e.~g. by postulating $1 \in S$ or $i
\leq g/2$). By $\pi_n\negthickspace: \Mgnbar \to \Mgbar{n-1}$ we will denote
the forgetful map which forgets the $n$-th point, while by $\pi_{(jk \mapsto
\bullet)}$ we mean the map which identifies the divisor $\Delta_{0:jk}
\subseteq \Mgnbar$ with $\Mgbar{n-1}$ by removing the rational component and
introducing the new marking $\bullet$ for the former point of attachment.
By $\pi\negthickspace: \mathcal{M}_{g,1} \times_{\mathcal{M}_g} \Mgn =: \U
\to \Mgn$ we denote the universal family over $\Mgn$ with sections
$\sigma_1,\, \dots,\, \sigma_n \negthinspace: \Mgn \to \U$, and by $\omega_\pi
\in \Pic(\U/\Mgn)$ the relative dualizing sheaf of the map $\pi$. Picard
groups are always understood in the functorial sense, i.~e. as groups of
divisor classes on the moduli stack.

\subsection*{Acknowledgements}
This work is part of my PhD thesis. I am very grateful to my advisor Gavril
Farkas for many helpful discussions and comments. My thanks also go to the
referee for a detailed reading and numerous suggestions for improvement.
I am supported by the DFG Priority Project SPP 1489.

\section{Preliminaries}
\label{sec:preliminaries}

\subsection{The Picard group of $\Mgnbar$}
\label{ssec:picard_group}
We quickly recall the well-known description of $\Pic(\Mgnbar)$. The
pushforward $\mathbb{E} := \pi_* \omega_\pi$ of $\omega_\pi$ to $\Mgn$ is
called the \emph{Hodge bundle}. It is a vector bundle of rank $g$ whose
determinant line bundle is denoted by $\lambda := \bigwedge^g \E$. For $j =
1,\, \dots,\, n$, the pullback of $\omega_\pi$ via the section $\sigma_j$ is
denoted by $\psi_j := \sigma_j^* \omega_\pi$. Moreover, we denote by
$\delta_0$ the line bundle corresponding to irreducible nodal pointed
stable curves, and by $\delta_{i:S}$ the one corresponding to pointed stable
curves consisting of two components of genera $i$ and $g-i$ that meet at a
node, with the marked points indexed by $S$ lying on the former. In
\cite{bib:arbarello-cornalba-picard-groups} it is proven that $\Pic(\Mgnbar)$
is freely generated by $\lambda$, the $\psi_j$, $\delta_0$ and the
$\delta_{i:S}$.

\subsection{Limit linear series}
Throughout this paper, we will make extensive use of the theory of limit
linear series, as first developed by Eisenbud and Harris
\cite{bib:eisenbud-harris-lls}. Here we briefly recall the most important
concepts and results. Recall that a \emph{linear series} of degree $d$ and
dimension $r$ on a smooth curve $C$ (in short, a $g^r_d$) is given by a pair
$\ell = (\L,\, V)$, where $\L$ is a line bundle of degree $d$ on $C$ and $V
\subseteq H^0(C,\, \L)$ is a subspace of projective dimension $r$. The
\emph{vanishing sequence} $a^\ell(p) = (0 \leq a^\ell_0(p) < \dots <
a^\ell_r(p) \leq d)$ of $\ell$ at a point $p \in C$ is the set
$\left\{ \ord_p(\sigma) \,\big|\, \sigma \in V \right\}$ of vanishing orders
of sections of $\ell$, ordered ascendingly.

\begin{definition}
Let $C$ be a nodal curve of compact type with irreducible components $C_1,\,
\dots,\, C_s$ and $r$, $d$ natural numbers. A \emph{limit $g^r_d$} on $C$ is a
collection $\ell$ of linear series $\ell_i = (\L_i,\, V_i)$ of degree $d$ and
dimension $r$ on each component $C_i$, satisfying the compatibility conditions
\begin{equation*}
a^{\ell_i}_m(\nu) + a^{\ell_j}_{r-m}(\nu) \geq d,\qquad m = 0,\,
\dots,\, r
\end{equation*}
for each node $\nu$ at which the components $C_i$ and $C_j$ meet. The
$\ell_i$ are called the \emph{aspects} of $\ell$. A \emph{section} of $\ell$
is a collection $\sigma = (\sigma_1,\, \dots,\, \sigma_s)$ of sections
$\sigma_i \in V_i$ satisfying
the compatibility conditions
\begin{equation*}
\ord_\nu(\sigma_i) + \ord_\nu(\sigma_j) \geq d,\qquad m = 0,\, \dots,\, r
\end{equation*}
for each node $\nu$ at which $C_i$ and $C_j$ meet.
If $p \in C$ is a smooth point, the \emph{vanishing sequence} of $\ell$ at
$p$ and the \emph{vanishing order} of a section $\sigma$ of $\ell$ at $p$ are
respectively defined to be $a^\ell(p) := a^{\ell_i}(p)$ and $\ord_p(\sigma) :=
\ord_p(\sigma_i)$, where $C_i$ is the component of $C$ on which $p$ lies.
\end{definition}

The usefulness of the concept of limit linear series lies in the fact that
they are indeed limits of linear series: By \cite[Section
2]{bib:eisenbud-harris-lls}, if a nodal curve of compact type lies in the
closure of the locus of curves admitting a $g^r_d$, then it
admits a limit $g^r_d$, and for $r = 1$ the converse is also true (see
\cite[Proposition 3.1]{bib:eisenbud-harris-lls}). This result remains
true even if one prescribes fixed vanishing sequences at points specializing
to smooth points on the nodal curve.

We finally recall two well-known facts about linear series on curves: a
generic curve $C$ of genus $g$ has a $g^r_d$ if and only if the
\emph{Brill-Noether-number}
\begin{equation*}
\rho(g,\, r,\, d) = g - (r + 1) (g - d + r)
\end{equation*}
is non-negative, and postulating a vanishing sequence $a = (a_0,\, \dots,\,
a_r)$ at a generic point of $C$ imposes $\sum_{i=0}^r (a_i - i)$ conditions on
the space of $g^r_d$'s on $C$.

\subsection{Pushforward and pullback formulas}
For computing pullbacks of divisor classes, we need the following formulas,
which can be found in \cite[p. 161]{bib:arbarello-cornalba-picard-groups}:
\begin{lemma}
\label{lem:pullback_forgetful}
If $\pi_n: \Mgnbar \to \Mgbar{n-1}$ is the forgetful map forgetting the last
point, then we have the following formulas for pullbacks of divisor classes:
\begin{enumerate}
\item $\pi_n^* \lambda = \lambda,$
\item $\pi_n^* \psi_j = \psi_j - \delta_{0:jn},$
\item $\pi_n^* \delta_0 = \delta_0,$
\item $\pi_n^* \delta_{i:S} = \delta_{i:S} + \delta_{i:S \cup \{ n \}},$
except that $\pi_1^* \delta_{g/2:\emptyset} = \delta_{g/2:\emptyset}$.
\end{enumerate}
\end{lemma}

To apply the Grothendieck-Riemann-Roch formula in Section
\ref{sec:main_coefficients}, we need certain formulas for pushforwards of
intersections of cycles on the universal family, which can be found for
example in \cite[Lemma 3.13]{bib:farkas-mustata-popa}. We reproduce the ones
that concern us here:
\begin{lemma}
\label{lem:pushforward_universal}
With notation as given in Section \ref{sec:introduction},
\begin{enumerate}
\item $\pi_* \big( c_1(\omega_\pi)^2 \big) = 12 \lambda$,
\item $\pi_* \big( c_1(\omega_\pi) c_1(\sigma_j) \big) = \psi_j$, and
\item $\pi_* \big( c_1(\sigma_j)^2 \big) = -\psi_j$.
\end{enumerate}
\end{lemma}

In order to be able to apply a pushdown technique in Section
\ref{sec:boundary}, we also need various formulas for pushforwards of
intersections of basis divisor classes via the map $\pi_{(jk \mapsto
\bullet)}$ which identifies the divisor $\Delta_{0:jk}$ with $\Mgbar{n-1}$.
They can be found in a table in \cite[Theorem 2.8]{bib:logan}; we list the
relevant ones here:
\begin{lemma}
\label{lem:pushforward_forgetful}
The following formulas for pushforwards of intersection cycles hold:
\begin{enumerate}
\item $\mathrlap{\pi_{(1n \mapsto \bullet)*}(\lambda \cdot \delta_{0:1n})}
\phantom{\pi_{(1n \mapsto \bullet)*}(\delta_{i:S} \cdot \delta_{0:1n})} =
\lambda,$
\item $\mathrlap{\pi_{(1n \mapsto \bullet)*}(\psi_j \cdot \delta_{0:1n})}
\phantom{\pi_{(1n \mapsto \bullet)*}(\delta_{i:S} \cdot \delta_{0:1n})} =
\begin{cases}
0 & \text{for $j = 1,\, n$,}\\ \psi_j & \text{for $j = 2,\, \dots,\, n - 1$,}
\end{cases}$
\item $\mathrlap{\pi_{(1n \mapsto \bullet)*}(\delta_0 \cdot \delta_{0:1n})}
\phantom{\pi_{(1n \mapsto \bullet)*}(\delta_{i:S} \cdot \delta_{0:1n})} =
\delta_0,$
\item $\mathrlap{\pi_{(1n \mapsto \bullet)*}(\delta_{0:1n}^2)}
\phantom{\pi_{(1n \mapsto \bullet)*}(\delta_{i:S} \cdot \delta_{0:1n})} =
-\psi_\bullet,$
\item $\pi_{(1n \mapsto \bullet)*}(\delta_{i:S} \cdot \delta_{0:1n}) =
\begin{cases} \delta_{i:S} & \text{if $1,\, n \not\in S$,}\\
\delta_{i:S'} & \text{if $1,\, n \in S$,}\\
0 & \text{if $1 \in S,\, n \notin S$ or $1 \notin S,\, n \in S$,}
\end{cases}$\\ where $S' := \big( S \setminus \{ 1,\, n \} \big) \cup \{
\bullet \}$.
\end{enumerate}
\end{lemma}
The corresponding formulas for the pushforwards of intersections of divisors
with other boundary divisor classes of the form $\delta_{0:jk}$ can easily be
obtained from Lemma \ref{lem:pushforward_forgetful} by applying the
$S_n$-action permuting the points on $\Mgnbar$. Note that when we take out the
basis elements of $\Pic(\Mgnbar)$ that get mapped to $0$ in the above
formulas, the map $\alpha \mapsto \pi_{(1n \mapsto \bullet)*}(\alpha \cdot
\delta_{0:1n})$ is injective on the span of the remaining basis elements, a
fact we will make use of in Section \ref{sec:boundary} (see Remark
\ref{rmk:pushdown}).

Finally, for applying the pushdown technique we also need to know how the
divisor $\Dd$ behaves under intersection and pushforward:
\begin{lemma}
\label{lem:pushforward_Dd}
If $j,\, k \in [n]$ are two indices such that $d_j$ and $d_k$ have the same
sign, then
\begin{equation*}
\pi_{(jk \mapsto \bullet)*}(\Dd \cdot \delta_{0:jk}) = \D_{\d'},
\end{equation*}
where $\d' = (d_1,\, \dots,\, \widehat{d_j},\, \dots,\, \widehat{d_k},\,
\dots,\, d_n,\, d_\bullet = d_j + d_k)$.
\end{lemma}
\begin{proof}
This is an easy generalization of the proof of \cite[Proposition
5.3]{bib:logan}.
\end{proof}

\section{Computation of the main coefficients}
\label{sec:main_coefficients}
We write the class of the divisor $\Dd$ as
\begin{equation}
\label{eq:D_coefficients}
\left[ \Dd \right] = a \lambda + \sum_{j=1}^n c_j \psi_j + b_0 \delta_0 +
\sum_{i,S} b_{i:S} \delta_{i:S}.
\end{equation}
In this section we determine the coefficients $a$ and $c_j$ by
expressing $D_{\d}$ as the degeneracy locus of a map of vector bundles of the
same rank and applying Porteous' formula. These calculations will also be
instrumental in computing some of the boundary coefficients $b_0$ and
$b_{i:S}$ in Section \ref{sec:boundary}, while the remaining ones will be
obtained by intersecting the closure $\Dd$ with suitably chosen test curves.

The top Chern class $\lambda_g := c_g(\E)$ of the Hodge bundle is known to
have class $0$ in $A^g(\Mgn)$ (see \cite{bib:looijenga}). Therefore we can
find a nowhere vanishing section of $\E$, or equivalently, a relative section
of $\omega_\pi$ over $\Mgn$, whose zero locus cuts out a canonical divisor on
every fiber of $\pi$. We denote that zero locus by $\Kscr$. Furthermore, we
denote by $\Dscr := \sum_{j=1}^n d_j \sigma_j \in \Pic(\U/\Mgn)$ the relative
divisor which on every fiber cuts out the divisor given by the linear
combination of the marked points.

We now consider the restriction map $\rho \negthinspace:\, \omega_\pi(\Dscr)
\to \omega_\pi(\Dscr)\big|_\Kscr$ and its direct image
\begin{equation}
\label{eq:vector_bundle_map}
\phi := R^0\pi_*\rho \negthinspace:\,
R^0\pi_*\big(\omega_\pi(\Dscr)\big) \to
R^0\pi_*\big(\omega_\pi(\Dscr)\big|_\Kscr\big).
\end{equation}
Since $\Dscr$ has relative degree $g-1$, we find that
$R^1\pi_*(\omega_\pi(\Dscr)) = 0$. Similarly, $\omega_\pi(\Dscr)|_\Kscr$ is
torsion on fibers, so we also have $R^1\pi_*(\omega_\pi(\Dscr)|_\Kscr) = 0$.
Thus by Grauert's theorem, both sheaves in \eqref{eq:vector_bundle_map} are in
fact locally free, and by Riemann-Roch they are easily seen to both have rank
$2g-2$.

We are now in a position to compute the main coefficients of $\Dd$.

\begin{proposition}
\label{prop:main_coefficients}
In the expression $\eqref{eq:D_coefficients}$ for $\left[ \Dd \right]$, we
have $a = -1$ and $c_j = \binom{d_j + 1}{2}$.
\end{proposition}
\begin{proof}
The short exact sequence
\begin{equation}
\label{eq:degeneracy_SES}
0 \to \O_\U(\Dscr) \to \omega_\pi(\Dscr) \stackrel{\rho}{\to}
\omega_\pi(\Dscr)\big|_\Kscr \to 0,
\end{equation}
yields after pushing down the long exact sequence
\begin{equation}
\label{eq:degeneracy_LES}
\begin{split}
0 &\to R^0\pi_*\big(\O_\U(\Dscr)\big) \to R^0\pi_*\big(\omega_\pi(\Dscr)\big)
\stackrel{\phi}{\to} R^0\pi_*\big(\omega_\pi(\Dscr)\big|_\Kscr\big) \\
&\to R^1\pi_*\big(\O_\U(\Dscr)\big) \to 0.
\end{split}
\end{equation}
Since $\sum_{j=1}^n d_j = g - 1$ implies $h^0(C,\, \sum_{j=1}^n d_j x_j) =
h^1(C,\, \sum_{j=1}^n d_j x_j)$ for every point $[C;\, x_1,\, \dots,\, x_n]
\in \Mgn$, the sequence \eqref{eq:degeneracy_LES} stays exact after passing
to a fiber. Thus, the divisor $D_{\d}$ is exactly the degeneracy locus of the
map $\phi$, and by Porteous' formula it follows that
\begin{equation}
\label{eq:porteous_expression}
\left[ D_{\d} \right] =
c_1\big(R^0\pi_*\big(\omega_\pi(\Dscr)\big|_\Kscr\big)\big) -
c_1\big(R^0\pi_*\big(\omega_\pi(\Dscr)\big)\big).
\end{equation}
We can calculate the two terms in \eqref{eq:porteous_expression} by a
Grothendieck-Riemann-Roch computation. For the first one, we obtain
\begin{equation*}
\begin{split}
\ch\big(\pi_!\big(\omega_\pi(\Dscr)\big|_\Kscr\big)\big) &=
\ch\big(\pi_*\big(\omega_\pi(\Dscr)\big|_\Kscr\big)\big)\\
&= \pi_* \Big[ \ch\big(\omega_\pi(\Dscr)\big|_\Kscr\big) \cdot
\td\big(\omega_\pi^\vee\big) \Big]\\
&= \pi_* \Big[ \big(\ch\big(\omega_\pi(\Dscr)\big) -
\ch\big(\O_\U(\Dscr)\big)\big) \cdot
\td\big(\omega_\pi^\vee\big) \Big] \qquad \text{(by
\eqref{eq:degeneracy_SES})}\\
&= \pi_* \Big[ \big(\ch(\omega_\pi) - 1\big) \cdot \ch(\Dscr) \cdot
\td\big(\omega_\pi^\vee\big) \Big]\\
&= \pi_* \Big[ \big( c_1(\omega_\pi) + \frac{1}{2} c_1^2(\omega_\pi) +
\dots \big) \cdot \big( 1 + c_1(\Dscr) + \frac{1}{2} c_1^2(\Dscr) + \dots
\big) \cdot \\
& \phantom{=\pi_*} \cdot \big( 1 - \frac{1}{2} c_1(\omega_\pi) + \frac{1}{12}
c_1^2(\omega_\pi) + \dots \big) \Big] \\
&= (2g - 2) + \pi_* \Big[ c_1(\omega_\pi) c_1(\Dscr) + \dots \Big] \\
&= (2g-2) + \sum_{j=1}^n d_j \psi_j + \dots \qquad \text{(by Lemma
\ref{lem:pushforward_universal})},
\end{split}
\end{equation*}
while for the second one we compute
\begin{equation*}
\begin{split}
\ch\big(\pi_!\big(\omega_\pi(\Dscr)\big)\big) &=
\ch\big(\pi_*\big(\omega_\pi(\Dscr)\big)\big)\\
&= \pi_* \Big[ \ch(\omega_\pi) \cdot \ch(\Dscr) \cdot
\td\big(\omega_\pi^\vee\big) \Big]\\
&= \pi_* \Big[ \big( 1 + c_1(\omega_\pi) + \frac{1}{2} c_1^2(\omega_\pi) +
\dots \big) \cdot \big( 1 + c_1(\Dscr) + \frac{1}{2} c_1^2(\Dscr) + \dots
\big) \cdot \\
& \phantom{=\pi_*} \cdot \big( 1 - \frac{1}{2} c_1(\omega_\pi) + \frac{1}{12}
c_1^2(\omega_\pi) + \dots \big) \Big] \\
&= (2g-2) + \lambda + \frac{1}{2} \sum_{j=1}^n (d_j - d_j^2) \psi_j + \dots
\qquad \text{(by Lemma \ref{lem:pushforward_universal})}.
\end{split}
\end{equation*}
Putting these together into \eqref{eq:porteous_expression} yields the result.
\end{proof}

\section{Intersections with test curves}
\label{sec:test_curves}
For later use in Section \ref{sec:boundary}, we will gather here several
computations of intersections of $\Dd$ with families of pointed curves which
are wholly contained in the boundary of $\Mgnbar$. This constitutes the main
work in computing the class of $\Dd$, the remaining part being mainly a
properly engineered application of the results presented here.

\begin{remark}
\label{rmk:Schubert_fully_ramified}
In proving the results of this section, we will often come across questions of
the following form: Given a curve $C$ of genus $g$ and a positive integer $d$,
how many $g^1_d$'s $\ell$ are there on $C$ satisfying some ramification
conditions whose codimensions add up to $\rho(g,\, 1,\, d)$?

In our cases, among the conditions there will always be one of \emph{full
ramification}, where we require $\ell$ to contain some fixed effective
divisor $D$ of degree $d$. This reduces the problem to a Schubert calculus
computation in the Grassmannian $\mathbb{G}(1,\, r)$, where $r := r(D) =
h^0(C,\, D) - 1$. Postulating the vanishing sequence $(a,\, b)$ at a generic
point of $C$ corresponds to the Schubert cycle $\sigma_{a,b-1}$, and requiring
$\ell$ to contain $D$ amounts to intersecting with $\sigma_{r-1} :=
\sigma_{0,r-1}$. Since
\begin{equation*}
\sigma_{\alpha_1,\beta_1} \cdot \ldots \cdot \sigma_{\alpha_k,\beta_k} \cdot
\sigma_{r-1} = 1 \qquad \text{for } \sum_{i=1}^k (\alpha_i + \beta_i) = r - 1,
\end{equation*}
in such cases $\ell$ is always unique.
\end{remark}


We first consider the case $n = 2$, where we write $\d = (g + b - 1;\, -b)$
with $b > 0$. Here and in the following, the intersection numbers of the
families in question with generators of $\Pic(\Mgnbar)$ that are not
explicitly mentioned in the Lemmas are implied (and easily seen) to be 0.

\begin{lemma}
\label{lem:family_elliptic_tail}
Let $\left( C;\, x_1,\, x_2,\, y \right)$ be a generic 3-pointed curve of
genus $g - 1$, and let $F$ be the family in $\Mgbar{2}$ obtained by gluing
the marked point $y$ to a base point of a generic plane cubic pencil. Then we
have
\begin{align*}
& \mathrlap{F \cdot \Dd = 0,}\\
& F \cdot \lambda = 1, && F \cdot \delta_0 = 12, && F \cdot \delta_{g-1:12} =
-1.
\end{align*}
\end{lemma}
\begin{proof}
A member of $F$ lying in $\Dd$ has a limit $g^1_{d_1}$ whose $C$-aspect
$\ell_C$ is spanned by $d_1 x_1$ and $b x_2 + \sigma$ for some $\sigma
\in \linsys{d_1 x_1 - b x_2}$. By Riemann-Roch, $h^0(C,\, (g + b - 1) x_1 - b
x_2) = 1$ for $x_1,\, x_2$ generic, so $\ell_C$ is unique, and since $y$ is
also generic, it has vanishing sequence $a^{\ell_C}(y) = (0,\, 1)$. Thus the
aspect on the elliptic tail would have to have vanishing sequence $(d_1 - 1,\,
d_1)$ at the base point, which is impossible.

The remaining intersection numbers are well known and can be found e.~g. in
\cite[p. 173f.]{bib:harris-morrison}.
\end{proof}


\begin{lemma}
\label{lem:family_2_pos}
Let $\left( C;\, x_2 \right)$ be a generic $1$-pointed curve of genus $g$, and
let $F$ be the family in $\Mgbar{2}$ obtained by letting a point $x_1$ move
along $C$. Then we have
\begin{align*}
& \mathrlap{F \cdot \Dd = g (d_1^2 - 1),}\\
& F \cdot \psi_1 = 2g - 1, && F \cdot \psi_2 = 1, && F \cdot \delta_{0:12} =
1.
\end{align*}
\end{lemma}
\begin{proof}
We compute the intersection number $F \cdot \Dd$ by degenerating $C$ to a
\emph{comb curve} $R \cup_{y_1} E_1 \cup \dots \cup_{y_g} E_g$ consisting of a
rational spine $R$ to which are attached $g$ elliptic tails at generic points
$y_1,\, \dots,\, y_g$, with the point $x_2$ lying on $R$. As shown in
\cite[Section 9]{bib:eisenbud-harris-cuspidal-rational-curves}, the variety of
limit $g^r_d$'s is reduced on a generic such curve, so all we have to do is
count the number of limit linear series $\ell = (\ell_R,\, \ell_{E_1},\,
\dots,\, \ell_{E_g})$ of type $g^1_{d_1}$ satisfying the given vanishing
conditions at $x_1$ and $x_2$.

By \cite[Proposition 1.1]{bib:eisenbud-harris-lls}, we must have $x_1 \in E_i$
for some $i$. The $E_j$-aspect of each elliptic tail $E_j$ with $j \neq i$
must satisfy $a^{\ell_{E_j}}(y_j) \leq (d_1 - 2,\, d_1)$, giving
$a^{\ell_R}(y_j) \geq (0,\, 2)$ for these $j$. Thus the $R$-aspect of $\ell$
is a $g^1_{d_1}$ that contains the divisor $d_1 y_i$, vanishes to order $b$ at
$x_2$ and is simply ramified at $(g-1)$ further points, corresponding to the
Schubert cycle
\begin{equation*}
\sigma_{a_0^{\ell_R}(y_i), d_1 - 1} \cdot \sigma_{b-1} \cdot \sigma_1^{g-1}
\qquad \text{in } \mathbb{G}(1,\, d_1).
\end{equation*}
Counting dimensions, this is non-empty only if $a_0^{\ell_R}(y_i) = 0$, and
then $\ell_R$ is unique by Remark \ref{rmk:Schubert_fully_ramified}. We thus
get the upper bound $a^{\ell_R}(y_i) \leq (0,\, d_1)$, which by the
compatibility conditions is equivalent to $a^{\ell_{E_i}}(y_i) \geq (0,\,
d_1)$. Since also $a^{\ell_{E_i}}(x_1) \geq (0,\, d_1)$, this is possible only
if equality holds everywhere and $x_1 - y_i$ is a non-trivial $d_1$-torsion
point in $\Pic^0(E_i)$. Thus each of the $g$ elliptic tails gives exactly
$(d_1^2 - 1)$ possibilities for $x_1$.

The remaining intersection numbers can be found by standard techniques.
\end{proof}


\begin{lemma}
\label{lem:family_2_neg}
Let $\left( C;\, x_1 \right)$ be a generic $1$-pointed curve of genus $g$, and
let $F$ be the family in $\Mgbar{2}$ obtained by letting a point $x_2$ move
along $C$. Then we have
\begin{align*}
& \mathrlap{F \cdot \Dd = g (b^2 - 1),}\\
& F \cdot \psi_1 = 1, && F \cdot \psi_2 = 2g - 1, && F \cdot \delta_{0:12} =
1.
\end{align*}
\end{lemma}
\begin{proof}
We proceed as in the proof of Lemma \ref{lem:family_2_pos}, degenerating $C$
to a comb curve where now $x_1 \in R$. Reasoning as before, we find that
$x_2 \in E_j$ for some $j$ and $a^{\ell_R}(y_j) \leq (0,\, b)$ for dimension
reasons, so $a^{\ell_{E_j}}(y_j) \geq (g - 1,\, g + b - 1)$. Together with
$a^{\ell_{E_j}}(y_j) \geq (0,\, b)$ this implies that $x_2 - y_j$ is a
non-trivial $b$-torsion point in $\Pic^0(E_j)$, so each of the $g$ elliptic
tails contributes $(b^2 - 1)$ possibilities for $x_2$.
\end{proof}


\begin{lemma}
\label{lem:family_G_i:12_*}
Let $\left(C_1;\, x_1,\, x_2,\, y \right)$ be a generic $3$-pointed curve of
genus $g-i$, $C_2$ a generic curve of genus $i \geq 2$, and let $F$ denote the
family in $\Mgbar{2}$ obtained by gluing $y$ to a moving point of $C_2$. Then
we have
\begin{align*}
& F \cdot \Dd = i (i^2 - 1),\\
& F \cdot \delta_{g-i:12} = 2 - 2i.
\end{align*}
\end{lemma}
\begin{proof}
Let $\ell = (\ell_{C_1},\, \ell_{C_2})$ be a limit $g^1_{d_1}$ on $C$. By
genericity, the family of $g^1_{d_1}$'s on $C_1$ with the required
vanishing at $x_1$ and $x_2$ has dimension $\rho(g - i,\, 1,\, d_1) - (d_1 -
1) - (b - 1) = i - 1$, so for $y \in C_1$ also generic we must have
$a_1^{\ell_{C_1}}(y) \leq i$. The compatibility relations then force
$a_0^{\ell_{C_2}}(y) \geq d_1 - i$. Since $\ell_{C_2}$ contains the divisor
$d_1 y$, this means that $\linsys{i y}$ is a $g^1_i$ on $C_2$, i.~e. $y$ is
one of the $i (i^2 - 1)$ Weierstra\ss{} point of $C_2$. Since $C_2$ is
generic, it has only ordinary Weierstra\ss{} points, so we must have equality,
and $\ell$ is unique by Remark \ref{rmk:Schubert_fully_ramified}.
\end{proof}


We now turn to cases where $n = 3$. We will first suppose that $d_1,\, d_2 >
0$, while $d_3 < 0$, and we write $b := -d_3$ and $d := d_1 + d_2 = g + b -
1$.

For $b = 1$, the following result was already proven in \cite[Proposition
3.3]{bib:logan} and \cite[Lemma 6.2]{bib:diaz-thesis}.

\begin{lemma}
\label{lem:family_F2_2pos}
Let $\left( C;\, x_2,\, x_3 \right)$ be a generic $2$-pointed curve of genus
$g$, and let $F$ be the family in $\Mgbar{3}$ obtained by letting a point
$x_1$ vary on $C$. Then we have
\begin{align*}
& \mathrlap{F \cdot \Dd = g d_1^2 - (g - d_2)_+,}\\
& F \cdot \psi_1 = 2g, && F \cdot \psi_2 = 1, && F \cdot \psi_3 = 1,\\
& F \cdot \delta_{0:12} = 1, && F \cdot \delta_{0:13} = 1.
\end{align*}
\end{lemma}
\begin{proof}
Suppose first that $g = 1$, i.~e. $b = d$. Then a $g^1_d$ containing the
divisors $d_1 x_1 + d_2 x_2$ and $d x_3$ exists if and only if these
are linearly equivalent, and since $d_2 > 0$ this gives $d_1^2$ possibilities
for $x_1$ as claimed.

If $g > 1$, we degenerate $C$ to a transverse union $C = E \cup_y C'$ such
that $\left( E;\, x_2,\, y \right)$ is a generic $2$-pointed elliptic curve
and $\left( C';\, y,\, x_3 \right)$ is a generic $2$-pointed curve of genus $g
- 1$. Then there is a decomposition $F = F_E + F_{C'}$ of $1$-cycles on
$\Mgbar{3}$, where $F_E$ and $F_{C'}$ correspond to the cases $x_1 \in E$ and
$x_1 \in C'$. These are in a natural way pushforwards via gluing morphisms of
$1$-cycles $F_E'$ and $F_{C'}'$ on $\Mbar{1,3}$ and $\Mbar{g-1,3}$,
respectively. We will show that
\begin{align}
\label{eq:FEdotDd}
F_E \cdot \Dd &= F_E' \cdot \D_{(d_1, d_2; -d)} \text{\quad(}= d_1^2 \text{
by the above)}, \\
\label{eq:FCdotDd}
F_{C'} \cdot \Dd &=
\begin{cases}
F_{C'}' \cdot \D_{(d_1, d_2 - 1; d_3)} & \text{if } d_2 > 1, \\
(g - 1)(d_1^2 - 1) & \text{if } d_2 = 1,
\end{cases}
\end{align}
and by induction we conclude that
\begin{equation*}
F \cdot \Dd = \sum_{i=1}^{d_2} d_1^2 + \sum_{i=d_2+1}^g (d_1^2 - 1) = g d_1^2
- (g - d_2)_+.
\end{equation*}

For showing \eqref{eq:FEdotDd}, let $\ell = (\ell_E,\, \ell_{C'})$ be a
$g^1_d$ having the required vanishing. Then $\ell_E$ has a section not
vanishing at $y$, so by the compatibility conditions $\ell_{C'}$ must be
totally ramified there. Counting dimensions as in the proof of Lemma
\ref{lem:family_G_i:12_*}, we find that the latter cannot have a base point at
$y$, so again by compatibility $\ell_E$ needs to have a section vanishing to
order $d$ at $y$. This is equivalent to requiring $(E;\, x_1,\, x_2,\, y)$
to lie in $\D_{(d_1,d_2;-d)}$.

Now consider \eqref{eq:FCdotDd}. Since $\ell_E$ contains $d_1 y + d_2 x_2$,
and by genericity $d_2 x_2 \not\equiv d_2 y$, it cannot also contain the
divisor $d y$, so we must have $a_1^{\ell_{E}}(y) \leq d - 1$. By the
compatibility condition then $a_0^{\ell_{C'}}(y) \geq 1$, and after removing
the base point we obtain a $g^1_{d-1}$ on $C'$ containing the divisor $d_1 x_1
+ (d_2 - 1) y$ and having a section vanishing to order $b$ at $x_3$. For $d_2
> 1$ this is equivalent to $(C';\, x_1,\, y,\, x_3) \in \D_{(d_1, d_2-1;
d_3)}$, while for $d_2 = 1$ the answer is given in Lemma
\ref{lem:family_2_pos}.
\end{proof}


\begin{lemma}
\label{lem:family_F2_i_2pos}
Let $\left( C_1;\, y \right)$ be a generic $1$-pointed curve of genus $i \geq
1$ , $\left( C_2;\, x_2,\, x_3,\, y \right)$ a generic $3$-pointed curve of
genus $g - i$, $\left( C = C_1 \cup_y C_2;\, x_2,\, x_3 \right)$ the
$2$-pointed curve obtained by gluing $C_1$ and $C_2$ at $y$, and $F$ the
family in $\Mgbar{3}$ obtained by letting a point $x_1$ move along $C_1$. Then
we have
\begin{align*}
& \mathrlap{F \cdot \Dd = i (d_1^2 - 1) + (i - d_1)_+,}\\
& F \cdot \psi_1 = 2i - 1, && F \cdot \delta_{i:1} = -1, && F \cdot
\delta_{i:\emptyset} = 1.
\end{align*}
These formulas also hold for $d_2 = 0$.
\end{lemma}
\begin{proof}
Let $\ell = (\ell_{C_1},\, \ell_{C_2})$ be a limit $g^1_d$ on $C$ satisfying
the given vanishing conditions and write $\ell_{C_2} = a_0 y + \ell_{C_2}'$,
where $a_0 := a_0^{\ell_{C_2}}(y)$. Then $\ell_{C_2}'$ contains the divisor
$(d_1-a_0)y + d_2 x_2$ and has a section vanishing to order $b$ at $x_3$, so
it corresponds to the Schubert cycle $\sigma_{r-1} \cdot \sigma_{b-1}$ in
$\mathbb{G}(1,\, r)$, where
\begin{equation*}
r := h^0(C_2,\, (d_1-a_0)y + d_2 x_2) - 1 = d - a_0 - g + i
\end{equation*}
by Riemann-Roch. This is non-empty only if $b \leq r$, or equivalently if $a_0
\leq i - 1$.

In case $a_0 < d_1$, we have $a_1^{\ell_{C_2}}(y) = d_1$, so
$a_0^{\ell_{C_1}}(y) \geq d_2$. Then $\ell_{C_1}' := \ell_{C_1} - d_2 y$ is a
$g^1_{d_1}$ fully ramified at $x_1$. Since $C_1$ is generic and therefore has
only ordinary Weierstra\ss{} points, this is possible only if $d_1 \geq i$.
Since $a_1^{\ell_{C_1}}(y) \geq d - a_0 \geq d - i + 1$, $\ell_{C_1}'$
vanishes to order $d_1 - i + 1$ at $y$, so by Lemma \ref{lem:family_2_pos} the
number of possibilities is
\begin{equation*}
F \cdot \Dd = i (d_1^2 - 1).
\end{equation*}

If on the other hand $a_0 = d_1$, we have $d_1 \leq i - 1$ by the above. By
another Schubert cycle computation for $\ell_{C_2}'$ we find that we need to
have $a_1^{\ell_{C_2}}(y) \leq i$, so $a_0^{\ell_{C_1}}(y) \geq d - i$. Thus
$\ell_{C_1} - (d-i) y$ is now a $g^1_i$ having $d_1 x_1 + (i - d_1) y$ as a
section. Applying Lemma \ref{lem:family_F2_2pos} with $\d = (d_1,\, i - d_1;\,
-1)$, we find that
\begin{equation*}
F \cdot \Dd = i d_1^2 - d_1.
\end{equation*}

Both arguments also go through when $d_2 = 0$.
\end{proof}


Still considering cases where $n = 3$, we now suppose that $d_1 > 0$, while
$d_2,\, d_3 < 0$, and we write $b_j := -d_j$ for $j = 2,\, 3$ and $b :=
-d_2 - d_3$, so that $d_1 = g + b - 1$.

\begin{lemma}
\label{lem:family_F2_2neg}
Let $\left( C;\, x_1,\, x_3 \right)$ be a generic $2$-pointed curve of genus
$g$, and let $F$ be the family in $\Mgbar{3}$ obtained by letting a point
$x_2$ vary on $C$. Then we have
\begin{align*}
& \mathrlap{F \cdot \Dd = g d_2^2,}\\
& F \cdot \psi_1 = 1, && F \cdot \psi_2 = 2g, && F \cdot \psi_3 = 1,\\
& F \cdot \delta_{0:12} = 1, && F \cdot \delta_{0:23} = 1.
\end{align*}
\end{lemma}
\begin{proof}
This is similar to the proof of Lemma \ref{lem:family_F2_2pos}: Let $C = E
\cup_y C'$ again with now $x_1 \in E$ and $x_3 \in C'$. Then $F = F_E +
F_{C'}$ with
\begin{align*}
F_E \cdot \Dd &= F_E' \cdot \D_{(d_1;d_2,d_3-g+1)} = d_2^2 \quad \text{and}\\
F_{C'} \cdot \Dd &= F_{C'}' \cdot \D_{(d_1-1;d_2,d_3)}.
\end{align*}
The only difference to before is that now $\ell_{C'}$ has a $b_2$-fold base
point at $y$ in case $x_2 \in E$. The result follows by induction.
\end{proof}


\begin{lemma}
\label{lem:family_F2_i_2neg}
Let $\left( C_1;\, y \right)$ be a generic $1$-pointed curve of genus $i \geq
1$, $\left( C_2;\, x_1,\, x_3,\, y \right)$ a generic 3-pointed curve of genus
$g - i$, $\left( C = C_1 \cup_y C_2;\, x_1,\, x_3 \right)$ the $2$-pointed
curve obtained by gluing $C_1$ and $C_2$ at $y$, and $F$ the family in
$\Mgbar{3}$ obtained by letting a point $x_2$ move along $C_1$. Then we have
\begin{align*}
& \mathrlap{F \cdot \Dd = i (d_2^2 - 1),}\\
& F \cdot \psi_2 = 2i - 1, && F \cdot \delta_{i:2} = -1, && F \cdot
\delta_{i:\emptyset} = 1.
\end{align*}
These formulas also hold for $d_3 = 0$.
\end{lemma}
\begin{proof}
Let $\ell = (\ell_{C_1},\, \ell_{C_2})$ be a limit $g^1_{d_1}$ on $C$
satisfying the given vanishing conditions. Then $\ell_{C_2}$ must include the
divisor $d_1 x_1$, so $a_0^{\ell_{C_2}}(y) = 0$. The section of $\ell_{C_2}$
vanishing to order $b_3$ at $x_3$ must also vanish to order $a_1 :=
a_1^{\ell_{C_2}}(y)$ at $x_1$: otherwise the corresponding section of
$\ell_{C_1}$ would have to be fully ramified at $y$ while at the same time
vanishing to order $b_2$ at $x_2$, which is absurd. We thus need
\begin{equation*}
h^0( C_2,\, d_1 x_1 - a_1 y - b_3 x_3 ) = b_2 + i - a_1 \geq 1 \iff a_1 \leq
b_2 + i - 1,
\end{equation*}
where we used Riemann-Roch and the genericity of the points on $C_2$. By
compatibility, $a_0^{\ell_{C_1}}(y) \geq g - i + b_3$, and thus $\ell_{C_1} -
(g - i + b_3) y$ is a $g^1_{i+b_2-1}$ which is fully ramified at $y$ and
vanishes to order $b_2$ at $x_2$. By Lemma \ref{lem:family_2_neg} there are $i
(d_2^2 - 1)$ possibilities for $x_2$.
\end{proof}


We now finally consider the situation $n = 4$ with $d_1,\, d_2 > 0$ and
$d_3,\, d_4 < 0$. We write $b_j := -d_j$ for $j = 3,\, 4$ and $b := b_3 +
b_4$, so that $d := d_1 + d_2 = g + b - 1$.

\begin{lemma}
\label{lem:family_F_13_on_i}
Let $\left( C_1;\, x_1,\, x_3 \right)$ be a generic $2$-pointed curve of
genus $i$ with $1 \leq i \leq g$, $\left( C_2;\, x_2,\, x_4,\, y \right)$ a
generic $3$-pointed curve of genus $g-i$, and let $F$ be the family in
$\Mgbar{4}$ obtained by gluing $y$ to a moving point of $C_1$. Then we have
\begin{align*}
& \mathrlap{F \cdot \Dd = i (d_1 + d_3 - i + 1)^2 - (i - d_1)_+,}\\
& F \cdot \psi_1 = 1, && F \cdot \psi_3 = 1,\\
& F \cdot \delta_{i:13} = -2i, && F \cdot \delta_{i:1} = 1, && F \cdot
\delta_{i:3} = 1.
\end{align*}
\end{lemma}
\begin{proof}
Let $\ell = (\ell_{C_1},\, \ell_{C_2})$ be a limit $g^1_d$ on $C$
satisfying the given vanishing conditions. Then $\ell$ contains the divisor
$(d_1 x_1 + d_2 y,\, d_1 y + d_2 x_2)$.

Suppose first that $d_1 + d_3 \geq i$, or $d_2 + d_4 \leq g - i - 1$. Then
the base locus of $\ell_{C_2}$ cannot contain $d_1 y$, since $h^0(C_2,\, d_2
x_2 - b_4 x_4) = 0$ by Riemann-Roch and genericity.
Hence $a_1^{\ell_{C_2}}(y) = d_1$, and by a dimension count
$a_0^{\ell_{C_2}}(y) \leq i - d_3 - 1$, with equality attained for a unique
$g^1_d$. Thus $a^{\ell_{C_1}}(y) \geq (d_2,\, g - i - d_4)$, and we
can apply Lemma \ref{lem:family_F2_2neg} with $\d = (d_1;\, d_2 + d_4 - g +
i,\, d_3)$ to find
\begin{equation*}
F \cdot \Dd = i (d_2 + d_4 - g + i)^2.
\end{equation*}

If $d_1 + d_3 < i - 1$, then $h^0(C_1,\, d_1 x_1 - b_3 x_3) = 0$, so $d_2
y$ cannot be in the base locus of $\ell_{C_1}$, forcing $a_1^{\ell_{C_1}}(y) =
d_2$ and thus $a_0^{\ell_{C_2}}(y) = d_1$.
As in the proof of Lemma \ref{lem:family_F2_i_2neg}, we find that
$a^{\ell_{C_2}}(y) \leq (d_1,\, i - d_3 - 1)$, so $a^{\ell_{C_1}}(y) \geq (g -
i - d_4,\, d_2)$. Applying Lemma \ref{lem:family_F2_2pos} with $\d = (d_2 +
d_4 - g + i,\, d_1;\, d_3)$ then gives
\begin{equation*}
F \cdot \Dd = i (d_2 + d_4 - g + i)^2 - (i - d_1)_+.
\end{equation*}

Finally, if $d_1 + d_3 = i - 1$ we obtain $a^{\ell_{C_2}}(y) = (d_1,\, d_1 +
1)$ and $\ell_{C_2} - d_1 y$ must have a section vanishing to order $1$ at $y$
and $b_4$ at $x_4$. Since $h^0(C_2,\, d_2 x_2 - y - b_4 x_4) = 0$, this is
impossible, so in this case $F \cdot \Dd = 0$, which is consistent with the
other two formulas.
\end{proof}


\section{Computation of the boundary coefficients}
\label{sec:boundary}
For computing the boundary coefficients of $\Dd$ we will use a bootstrapping
approach, considering first the easiest non-trivial case $n = 2$, then
generalizing to the case $n > 2$ with exactly one $d_j < 0$, and finally
tackling the most general situation.

\subsection{The case $n = 2$}
For ease of notation, we will write $\d = (d_1,\, d_2) = (g + b - 1,\, -b)$
with $b \geq 1$ and denote the corresponding divisor by $\D_{(g+b-1,b)} =:
\D_b$.

\begin{proposition}
\label{prop:case_n=2}
The class of $\D_b$ is given by
\begin{equation*}
\begin{split}
\left[ \D_b \right] =& -\lambda + \binom{g + b}{2} \psi_1 + \binom{b}{2}
\psi_2 - 0 \cdot \delta_0 - \binom{g + 1}{2} \delta_{0:12}\\
& -\sum_{i=1}^{g-1} \left[ \binom{g - i + b}{2} \delta_{i:1} + \binom{g - i +
1}{2} \delta_{i:12} \right].
\end{split}
\end{equation*}
\end{proposition}
\begin{proof}
From Section \ref{sec:main_coefficients} we know that $a = -1$, $c_1 =
\binom{g + b}{2}$ and $c_2 = \binom{-b + 1}{2} = \binom{b}{2}$. Intersecting
$\D_b$ with the family from Lemma \ref{lem:family_G_i:12_*}, we find that
$b_{g-i:12} = -\binom{i + 1}{2}$, or dually $b_{i:12} = -\binom{g - i + 1}{2}$
for $i = 0,\, \dots,\, g-2$. From the family in Lemma
\ref{lem:family_F2_i_2pos} (taking $d_2 = 0$), we get
\begin{equation*}
b_{i:1} = (2i - 1) c_1 + b_{g-i:12} - i ((g + b - 1)^2 - 1) = -\binom{g - i
+ b}{2} \quad \text{for } i = 2,\, \dots,\, g-1,
\end{equation*}
and Lemma \ref{lem:family_F2_i_2neg} with $d_3 = 0$ gives $b_{g-1:12}
= -1$. Using Lemma \ref{lem:family_F2_i_2pos} once more, we get the value for
$b_{1:1}$, while finally Lemma \ref{lem:family_elliptic_tail} leads to $b_0 =
(b_{g-1:12} - a)/12 = 0$.
\end{proof}

\begin{remark}
\label{rmk:weierstrass_divisor}
Note that when we pull back from $\Mgbar{1}$ the Weierstra\ss{} divisor $\Wg$,
whose class is given in \eqref{eq:weierstrass_divisor_class}, we get by Lemma
\ref{lem:pullback_forgetful} that
\begin{equation*}
\begin{split}
\left[ \pi_2^* \Wg \right] &= -\lambda + \binom{g + 1}{2} \psi_1 - \binom{g +
1}{2} \delta_{0:12} - \sum_{i=1}^{g-1} \binom{g - i + 1}{2} \big( \delta_{i:1}
+ \delta_{i:12} \big) \\
&= \left[ \D_1 \right]
\end{split}
\end{equation*}
as expected. Furthermore it is easy to see that a 2-pointed curve $(C = C'
\cup_y \P^1;\, x_1,\, x_2)$ with $x_1,\, x_2 \in \P^1$ is in $\D_b$ exactly
when it has a limit $g^1_{g+b-1}$ whose $C'$-aspect satisfies
$a^{\ell_{C'}}(y) = (b - 1,\, g + b - 1)$, which is the case if and only if
$y$ is a Weierstra\ss{} point of $C'$. From Lemma
\ref{lem:pushforward_forgetful} we obtain accordingly
\begin{equation*}
\pi_{(12 \mapsto \bullet)*} (\left[ \D_b \right] \cdot \delta_{0:12}) =
-\lambda + \binom{g + 1}{2} \psi_\bullet - \sum_{i=1}^{g-1} \binom{g - i +
1}{2} \delta_{i:\bullet} = \left[ \Wg \right].
\end{equation*}
\end{remark}

\subsection{The case of exactly one negative $d_j$}
We now consider the next simplest case where exactly one of the $d_j$ is
negative (for definiteness, and without loss of generality, we take $d_n <
0$).

\begin{remark}
\label{rmk:pushdown}
Here and in the next section we will several times apply a ``pushdown''
argument which runs as follows: Let $j,\, k \in [n]$ be two indices such that
$d_j$ and $d_k$ have the same sign, and suppose that $\alpha \in
\Pic(\Mgnbar)$ is one of the basic divisor classes described in Section
\ref{ssec:picard_group} satisfying $\beta := \pi_{(jk \mapsto
\bullet)*}(\alpha \cdot \delta_{0:jk}) \neq 0$. Since then no other basis
element is mapped to $\beta$ and $\pi_{(jk \mapsto \bullet)*}(\left[ \Dd
\right] \cdot \delta_{0:jk}) = \left[ \D_{\d'} \right]$ with $\d'$ as in Lemma
\ref{lem:pushforward_Dd}, the coefficient of $\alpha$ in the expression for
$\Dd$ is the same as the coefficient of $\beta$ in the class of $\D_{\d'}$.
\end{remark}

\begin{proposition}
\label{prop:case_one_negative_dj}
If $d_j > 0$ for $j = 1,\, \dots,\, n-1$, then the class of $\Dd$ is given by
\begin{equation*}
\left[ \Dd \right] = -\lambda + \sum_{j=1}^n \binom{d_j + 1}{2} \psi_j - 0
\cdot \delta_0 - \sum_{i, S \subseteq [n-1]} \binom{\abs{d_S - i} + 1}{2}
\delta_{i:S}.
\end{equation*}
\end{proposition}
\begin{proof}
We already know from Section \ref{sec:main_coefficients} that $a = -1$ and
$c_j = \binom{d_j + 1}{2}$.

For the $b_{0:jk}$ with $j,\, k \in [n-1]$, we can apply the pushdown
argument explained in Remark \ref{rmk:pushdown} to the divisor class
$\delta_{0:jk}$ itself, which gets mapped to $-\psi_\bullet$. Thus we have
$b_{0:jk} = -c_\bullet$, where $c_\bullet$ is the coefficient of
$\psi_\bullet$ in the expression for $\pi_{(jk \mapsto \bullet)*}\big(\Dd
\,\cdot\, \delta_{0:jk}\big)$. Since $j,\, k \leq n-1$, we have $d_j,\, d_k >
0$, so we can apply Lemma \ref{lem:pushforward_Dd} to find that $b_{0:jk} =
-\binom{d_j + d_k + 1}{2}$. Similarly, in order to compute $b_{0:S}$ for $S
\subseteq [n-1]$, we can intersect with one divisor $\delta_{0:jk}$ with
$j,\, k \in S$ at a time and push down via the appropriate forgetful maps; by
inductively reasoning as before we find
\begin{equation*}
b_{0:S} = -\binom{d_S + 1}{2} \qquad \text{for } S \subseteq [n - 1].
\end{equation*}

Looking at Lemma \ref{lem:pushforward_forgetful} and using a simple induction
again, we see that when we successively let all of the points $x_1,\, \dots,\,
x_{n-1}$ come together and push down via the appropriate forgetful maps, the
divisor $\delta_{i:\emptyset}$ is mapped to $\delta_{i:\emptyset} =
\delta_{g-i:12}$ on $\Mgbar{2}$, so by Lemma \ref{lem:pushforward_Dd} and
Proposition \ref{prop:case_n=2} again we see that
\begin{equation*}
b_{i:\emptyset} = -\binom{i + 1}{2} \qquad \text{for } i \geq 1.
\end{equation*}

Next, using the test family from Lemma \ref{lem:family_F2_i_2pos} we get that
\begin{equation*}
b_{i:j} = (2i - 1) c_j + b_{i:\emptyset} - i (d_j^2 - 1) - (i - d_j)_+ =
-\binom{\abs{d_j - i} + 1}{2},
\end{equation*}
for $j \in [n - 1]$, and using a pushdown argument once again we arrive at
\begin{equation*}
b_{i:S} = -\binom{\abs{d_S - i} + 1}{2} \qquad \text{for } S \subseteq [n - 1]
\text{ and } i \geq 1.
\end{equation*}

Finally, the fact $b_0 = 0$ follows again from letting all of the points
$x_1,\, \dots,\, x_{n-1}$ coalesce, pushing down to $\Mgbar{2}$ and recurring
to Proposition \ref{prop:case_n=2}.
\end{proof}

\begin{remark}
\label{rmk:logan_divisor}
If $b = 1$, we expect $\Dd$ to be the pullback to $\Mgnbar$ of the divisor
\begin{equation*}
D = \left\{ \left[ C; x_1,\, \dots,\, x_{n-1} \right] \,\Big|\, h^0 \left(
C,\, d_1 x_1 + \dots + d_{n-1} x_{n-1} \right) \geq 2 \right\}
\end{equation*}
which was considered by Logan \cite{bib:logan}. Indeed, we have for $S
\subseteq [n-1]$ that
\begin{equation*}
b_{i:S \cup \{ n \}} = b_{g-i:[n-1] \setminus S} = -\binom{\abs{g - d_S -
g + i} + 1}{2} = b_{i:S},
\end{equation*}
and moreover $c_n = 0$ and $b_{0:jn} = -c_j$ for $j \in [n-1]$. Lemma
\ref{lem:pullback_forgetful} thus shows that
\begin{equation*}
\left[ D \right] = -\lambda + \sum_{j=1}^{n-1} \psi_j - 0 \cdot \delta_0 -
\binom{\abs{d_S - i} + 1}{2} \delta_{i:S},
\end{equation*}
which is consistent with the computations in \cite{bib:logan}.
\end{remark}

\subsection{The general case}
We will now finally deal with the most general case where there are at least
two $d_j$ of either sign, thereby proving formula \eqref{eq:class_Dd}. We
exclude the degenerate case where some $d_j$ equals $0$, since in this case
the divisor $\Dd$ is just a pullback of some $\D_{\d'}$ from some moduli space
with fewer marked points, so its class can easily be computed from Theorem
\ref{thm:class_Dd} with the help of the formulas in Lemma
\ref{lem:pullback_forgetful}.

\begin{theorem}
\label{thm:class_Dd}
The class of $\Dd$ in $\Pic(\Mgnbar)$ is given by
\begin{equation*}
\begin{split}
\left[ \Dd \right] =& -\lambda + \sum_{j=1}^n \binom{d_j + 1}{2} \psi_j -
0 \cdot \delta_0 \\
& - \sum_{\substack{i,\, S\\ S \subseteq S_+}}
\binom{\abs{d_S - i} + 1}{2} \delta_{i:S} - \sum_{\substack{i,\, S\\ S
\not\subseteq S_+}} \binom{d_S - i + 1}{2} \delta_{i:S}.
\end{split}
\end{equation*}
\end{theorem}
\begin{proof}
From Section \ref{sec:main_coefficients} we know that $a = -1$ and $c_j =
\binom{d_j + 1}{2}$. Using the by now familiar pushdown technique, we get from
Proposition \ref{prop:case_one_negative_dj} that $b_0 = 0$ and
\begin{equation*}
b_{i:S} = -\binom{\abs{d_S - i} + 1}{2} \qquad \text{for } S \subseteq S_+.
\end{equation*}
Thus we are left with computing the $b_{i:S}$ where the points indexed by
$S_-$ do not all lie on the same component.

Suppose first that $\emptyset \neq S \subsetneq S_-$. By letting the points
from $S_+$, $S$ and $S_- \setminus S$ respectively come together, we can
reduce to the case $n = 3$ with $d_1 = d_{S_+} > 0$, $d_2 = d_S < 0$ and $d_3
= d_{S_- \setminus S} < 0$. The divisor $\delta_{i:S}$ is mapped to
$-\psi_2$ for $i = 0$ and to $\delta_{i:2}$ for $i > 0$. We know that $c_2 =
\binom{d_2 + 1}{2}$, while for $i > 0$ we get from Lemma
\ref{lem:family_F2_i_2neg} that
\begin{equation*}
b_{i:2} = (2i - 1) \binom{d_2 + 1}{2} + b_{i:\emptyset} - i(d_2^2 - 1) =
-\binom{d_2 - i + 1}{2}.
\end{equation*}
Thus in total we deduce by Lemma \ref{lem:pushforward_Dd} that
\begin{equation*}
b_{i:S} = -\binom{d_S - i + 1}{2} \qquad \text{for } \emptyset \neq S
\subsetneq S_-.
\end{equation*}

Finally, let $S = S_1 \cup S_2$ with $\emptyset \neq S_1 \subsetneq S_+$ and
$\emptyset \neq S_2 \subsetneq S_-$. Letting the points from $S_1$, $S_+
\setminus S_1$, $S_2$ and $S_- \setminus S_2$ respectively come together, we
reduce to the computation of $b_{i:13}$ in the case $n = 4$. Taking the family
from Lemma \ref{lem:family_F_13_on_i} we find
\begin{equation*}
\begin{split}
b_{i:13} &= \frac{1}{2i} \left( c_1 + c_3 + b_{i:1} + b_{i:3} - i (d_1 + d_3 -
i + 1)^2 + (i - d_1)_+ \right)\\
&= -\binom{d_1 + d_3 - i + 1}{2}.
\end{split}
\end{equation*}
Note that although in Lemma \ref{lem:family_F_13_on_i} we require $i \geq 1$,
the above formula is invariant under the substitution $(i,\, d_1,\, d_3)
\mapsto (g - i,\, d_2,\, d_4)$, so it holds also for $i = 0$. Thus in total
we get
\begin{equation*}
b_{i:S} = -\binom{d_S - i + 1}{2} \qquad \text{for } S = S_+ \cup S_- \text{
with } \emptyset \neq S_1 \subsetneq S_+ \text{ and } \emptyset \neq S_2
\subsetneq S_-,
\end{equation*}
which finishes the computation of $\left[ \Dd \right]$.
\end{proof}

\bibliographystyle{amsalpha}
\bibliography{paper.bib}
\end{document}